\documentclass[reqno,12pt,letterpaper]{amsart}
\usepackage{amsmath,amssymb,amsthm,graphicx,mathrsfs,url}
\usepackage[usenames,dvipsnames]{color}
\usepackage[colorlinks=true,linkcolor=Red,citecolor=Green]{hyperref}

\def\?[#1]{\textbf{[#1]}\marginpar{\Large{\textbf{??}}}}
\newtheorem{prop}{Proposition}
\newtheorem{thm}[prop]{Theorem}

\newtheorem{conj}[prop]{Conjecture}
\newtheorem{rem}[prop]{Remark}

\numberwithin{equation}{section}
\numberwithin{prop}{section}

\renewcommand{\Re}{\mathop{\rm Re}\nolimits}
\renewcommand{\Im}{\mathop{\rm Im}\nolimits}

\DeclareMathOperator{\Press}{Pr}

\DeclareMathOperator{\supp}{supp}
\DeclareMathOperator{\tr}{tr}
\DeclareMathOperator{\Sp}{Sp}
\DeclareMathOperator{\Spec}{Spec}
\DeclareMathOperator{\Res}{Res}

\DeclareMathOperator{\PSL}{PSL}
\DeclareMathOperator{\GL}{GL}
\DeclareMathOperator{\Vol}{Vol}
\DeclareMathOperator{\KS}{KS}

\begin{document}
\title[Sublinear lower bounds]{Sublinear lower bounds of eigenvalues for twisted Laplacian on compact hyperbolic surfaces}
\author{Yulin Gong}
\email{gongyl22@mails.tsinghua.edu.cn}
\address{Department of Mathematical Sciences, Tsinghua University, Beijing 100084, China}
\author{Long Jin}
\email{jinlong@mail.tsinghua.edu.cn}
\address{Mathematical Sciences Center, Tsinghua University, Beijing, China \& Beijing Institute of Mathematical Sciences and Applications, Beijing, China}

\begin{abstract} 
We investigate the asymptotic spectral distribution of the twisted Laplacian associated with a real harmonic 1-form on a compact hyperbolic surface. In particular, we establish a sublinear lower bound on the number of eigenvalues in a sufficiently large strip determined by the pressure of the harmonic 1-form. Furthermore, following an observation by Anantharaman \cite{nalinideviation}, we show that quantum unique ergodicity fails to hold for certain twisted Laplacians.
\end{abstract}

\maketitle

\section{Introduction}
\label{s:intro}
Let $X=\mathbb{H}/\Gamma$ be a compact hyperbolic surface without boundary, $\Gamma$ a cocompact Fuchsian subgroup of $\PSL(2,\mathbb{R})$. We study the distribution for eigenvalues of the twisted Laplacian operators $\Delta_{\omega}$ on $X$ by a harmonic $1$-form $\omega \in \mathcal{H}^{1}(X,\mathbb{C})$, defined as follows:
\begin{equation}
\label{e:twisted-laplacian}
    \Delta_{\omega}f(x):=\Delta f(x)-2\langle \omega,df \rangle_{x}+|\omega|_{x}^{2}f(x),\quad\quad f=f(x)\in C^\infty(X).
\end{equation}
Here $\Delta$ is the usual Laplacian--Beltrami operator on $M$, $\langle\bullet,\bullet\rangle$ is the $\mathbb{C}$-bilinear form on $T_x^\ast X\otimes\mathbb{C}$ extending the Riemannian metric on $T_x^\ast X$, and $|\omega|_{x}^{2}=\langle \omega,\omega \rangle_{x}$. 

When $\omega\in\mathcal{H}^1(X,i\mathbb{R})$, the operator $\Delta_\omega$ is self-adjoint and related to the distribution of geodesics in a given homology class, see Phillips--Sarnak \cite{pshomology}, Katsuda--Sunada \cite{homology}. In this paper, we consider the situation that $\omega \in \mathcal{H}^{1}(X,\mathbb{R})$, that is, a real-valued harmonic $1$-form. Then $\Delta_{\omega}$ is a non-self-adjoint operator on $L^2(X)$ with discrete spectrum: 
\begin{equation}
\label{e:eigenvalues}
\Delta_{\omega}\phi_j+\lambda_j\phi_j=0, \quad 
\|\phi_j||_{L^2}=1 \quad\text{with}\quad
\lambda_0<\Re \lambda_1 \leq \Re \lambda_2\leq  \nearrow \infty.
\end{equation}
Anantharaman \cite{naliniearly} applies the twisted heat semi-groups \( \left\{e^{\frac{t\Delta_{\omega}}{2}}\right\}_{t\geq 0} \) to study the distribution of closed geodesics which are optimal in homology.

We use the spectral parameter $r_j\in\Sp(\omega)=\Sp(X,\omega)$ with $\Im r_j\geq0$ which is related to $\lambda_j\in\Spec(-\Delta_\omega)$ by the relation
$$\lambda_j=\frac{1}{4}+r_j^2,\quad j\in\mathbb{N}.$$

Our main theorem shows a sublinear lower bound of the spectral distribution away from the real axis for the twisted Laplacian on a compact hyperbolic surface. 
Let us define the following counting function for the eigenvalues of twisted Laplacian $\Delta_\omega$:
\begin{equation}
\label{e:counting-ar}
N_A(R):=\#\{r\in\Sp(\omega):|\Re r|\leq R, \ \Im r\geq A\},\quad A, \ R\geq 0.
\end{equation}

\begin{thm}
\label{t:sublinear}
If $\beta\in(0,1)$ and
\begin{equation}
\label{e:range-A}
0<A<\|\omega\|_s-\frac{1}{2}-\frac{\Press(\omega)-\|\omega\|_s}{1-\beta}
\end{equation}
there exist constants $C$ and $R_0>0$, depending on $\beta$ and $A$ such that for any $R\geq R_0$, we have:
\begin{equation}
\label{e:sublinear}
N_A(R)\geq\frac{1}{C}R^{\beta}.
\end{equation}
\end{thm} 
Here, we regard $\omega\in \mathcal{H}^{1}(X,\mathbb{R})$ as a function on the cosphere bundle $S^\ast M$ by
$$\omega(x,\xi) = \langle \omega, \xi \rangle_x.$$
$\Press(\omega)$ and $\|\omega\|_s$ are the pressure and the stable norm of $\omega(x,\xi)$ under the geodesic flow of $M$, respectively, see \eqref{e:pressure} and \eqref{e:stablenorm} for the definition. We note that the condition \eqref{e:range-A} requires $\omega$ to be sufficiently large for fixed $\beta\in(0,1)$. 

We define the essential spectral gap for the twisted Laplacian $\Delta_\omega$ as 
\begin{equation}
\label{e:essgap}
G_{\omega}=G_{X,\omega}:=\limsup\limits_{r\in\Sp(\omega), \ |\Re r|\to+\infty}\Im r,
\end{equation} 
or equivalently,
$$G_\omega:=\inf\{A>0: N_A(R)=\mathcal{O}(1), \ R\to+\infty\}.$$
Then Theorem \ref{t:sublinear} implies the following lower bound: 
\begin{equation}
\label{e:essgap-lbweak}
G_\omega\geq 2\|\omega\|_s-\Press(\omega)-\frac{1}{2}.
\end{equation}
In Remark \ref{generalizetohigherdimension}, we explain how to generalize the lower bound \eqref{e:essgap-lbweak} to higher-dimensional non-unitary representations with non-negative traces.

Our second theorem gives a different lower bound for $G_\omega$:
\begin{thm}
\label{t:essgap-lb}
\begin{equation}
\label{e:essgap-lb}
G_{\omega} \geq \frac{\Press(2\omega)-1}{2}-(2\Press(\omega)-\Press(2\omega))=\frac{3}{2}\Press(2\omega)-2\Press(\omega)-\frac{1}{2}.
\end{equation}
\end{thm} 
\begin{rem}
There is also a sublinear growth of the form \eqref{e:sublinear} from the proof of \eqref{e:essgap-lb}, but we only manage to obtain for $\beta\in(0,\frac{1}{2})$, see Section \ref{s:essgap}.
\end{rem}

For compact arithmetic surfaces given by a quaternion algebra, Anantharaman \cite[Corollary 1]{nalinideviation} proves sublinear growth \eqref{e:sublinear} with the following range
\begin{equation}
\label{e:sublinear-arithmetic}
0<A<\Press(\omega)-\frac{3}{4}-\frac{1}{2(1-\beta)},
\end{equation}
and thus 
\begin{equation}
\label{e:essgap-arithmetic}
G_\omega\geq\Press(\omega)-\frac{5}{4}.
\end{equation}

The twisted Selberg zeta function is defined as 
\begin{equation}
\label{e:twisted-zeta}
Z_\omega(s)=Z_{X,\omega}(s):=\prod_{k=0}^\infty\prod_{\gamma\in\mathcal{P}(X)}\left(1-e^{\int_\gamma\omega}e^{-(s+k)\ell_\gamma}\right),\quad \Re s\gg 1.
\end{equation}
Here $\mathcal{P}(X)$ is the set of oriented prime geodesics $\gamma$ and $\ell_\gamma$ is the length of $\gamma$. As the usual Selberg zeta function \cite{selberg}, $Z_\omega(s)$ has a meromorphic continuation to $\mathbb{C}$ and the zeroes of the $Z_\omega$ are given by (see e.g. M\"{u}ller \cite{muller}, Frahm--Spilioti \cite{frahmspilioti} and Naud--Spilioti \cite{naudspilioti})
\begin{itemize}
\item the trivial zeroes at $-k$, with multiplicity $-(2k+1)\chi(X)$, $k\in\mathbb{N}$. Here $\chi(X)$ is the Euler characteristic of $X$.
\item the spectral zeroes at $\frac{1}{2}\pm ir_{j}$, where $r_j\in\Sp(X,\omega)$ with the same multiplicity.
\end{itemize}
Therefore the asymptotic version of Riemann hypothesis for $Z_\omega$ means $G_\omega=0$. Theorem \ref{t:sublinear}, in particular, \eqref{e:essgap-lb} implies that $G_\omega>0$ for $\omega$ large enough, i.e. the failure of the asymptotic Riemann hypotheses for $Z_\omega$. When $\omega=0$, i.e. the usual Selberg zeta function, of course $G_\omega=0$ as all eigenvalues of the usual Laplacian are real.  Moreover, based on an observation of Anantharaman \cite{nalinideviation}, we have the following result on the eigenfunctions in the high-frequency limit $\Re r \to \infty$:
\begin{thm}
\label{t:nonque}
If there exists a closed geodesic $\gamma$ such that 
\begin{equation}
\int_\gamma\omega>\frac{3}{2}\ell_\gamma,
\end{equation}
then $G_\omega>0$ and the quantum unique ergodicity fails for the twisted Laplacian $\Delta_\omega$.
\end{thm}
Here the quantum unique ergodicity (QUE) refers to the equidistribution of the eigenfunctions in both physical and momentum space in the semiclassical limit. For the usual Laplacian--Beltrami operators on compact hyperbolic surfaces or more general compact manifolds with negative curvature, Rudnick--Sarnak \cite{que} conjectured the quantum unique ergodicity of Laplacian eigenfunctions, based on the pioneer work of Shnirelman \cite{shnirelman}, also later work of Colin de Verdi\`{e}re \cite{CdV} and Zelditch \cite{zelditch} of a weaker version, called the quantum ergodicity theorem, which is the equidistribution for a density one subsequence of eigenfunctions. Lindenstrauss \cite{linden} proved the arithmetic version. Some recent developments include Anantharaman \cite{nalinientropy}, Riviere \cite{riviere}, Dyatlov--Jin \cite{meassupp}, Dyatlov--Jin--Nonnenmacher \cite{varfup}.


\subsection{Spectral distribution of damped wave operators}
\label{s:previous}
In the high-frequency limit $\Re r\to \infty$, the eigenvalue problem of the twisted Laplacian \eqref{e:eigenvalues} and the stationary damped wave equation (see e.g. \cite{lebeau}) 
\begin{equation}
\label{e:dampwave}
P(\tau)u:=(-\Delta-\tau^2-2i\tau a)u=0,\quad a\in C^\infty(M;[0,\infty)), \ |\Re\tau|\to+\infty
\end{equation}
can be unified as a semiclassical damped wave operator $P(z,h)=P+ihQ(z;h)$, $h\to0+$, where $P=-h^2\Delta$ and $Q=Q(z;h)$ is the first order ``damping" term. The spectral theory of such semiclassical damped wave operator is first developed by Sj\"{o}strand \cite{sjostrand} and applied to the twisted Laplacian by Anantharaman \cite{nalinideviation} and recent work of the first author \cite{gong}. Most results are formulated in terms of \eqref{e:dampwave} and here we reformulate some general results in the case of the twisted Laplacian \eqref{e:twisted-laplacian} on hyperbolic surfaces. Firstly all eigenvalues lie in a strip:
$$\Sp(\omega)\subset\left\{0\leq\Im r\leq \Press(\omega)-\frac{1}{2}\right\}$$
with the highest eigenvalue $r_0=i(\Press(\omega)-\frac{1}{2})$, thus $G_\omega\leq\Pr(\omega)-\frac{1}{2}$, see Schenck \cite{schenckpressure,schenck} for the "pressure gap" for \eqref{e:dampwave}.

Moreover, the number of eigenvalues satisfies the Weyl law (see Sj\"{o}strand \cite{sjostrand} as well as earlier work of Markus--Matsaev \cite{mama})
\begin{equation}
\label{e:weyllaw}
N_0(R)=\#\Sp(\omega)\cap\{|\Re r|\leq R\}=\frac{\mathrm{Vol}(X)}{4\pi}R^2+\mathcal{O}(R).
\end{equation}
Sj\"{o}strand \cite{sjostrand} proved that for \eqref{e:dampwave}, due to the ergodicity of the geodesic flow, the imaginary part of the eigenvalues (i.e. ``decay rate") concentrate near the average of the ``damping term". In the case of the twisted Laplacian, the average of $\omega(x,\xi)$ over $S^\ast X$ is 0 and we have for all $A>0$, $N_A(R)=o(R^2)$. Anantharaman \cite{nalinideviation} gives a better estimate  
\begin{equation}
\label{e:concentration}
N_A(R)=\mathcal{O}(R^{2-c}), \quad R\to\infty,
\end{equation}
where the constant $c=c(\omega,A)>0$ is explicitly given by some large deviation rate and maximal expansion rate of the geodesic flow. See also Naud--Spilioti \cite{naudspilioti} for an analogue of the Weyl law \eqref{e:weyllaw} and spectral deviation \eqref{e:concentration} in the setting of higher-dimensional non-unitary representations. On the other hand, the first author \cite{gong} gives a better width for the spectral concentration: there exists $c(\omega,X)>0$ such that for any $0<c<c(\omega, X)$ and $\varepsilon>0$,
$$\#\Sp(\omega)\cap\{|\Re r|\leq R, \Im r\geq (\log R)^{-\frac{1-\varepsilon}{2}}\}=\mathcal{O}\left(\frac{R^2}{e^{c(\log R)^{\varepsilon}}(\log R)^{\varepsilon-1}}\right).$$
We also mention the recent work \cite{gongrandom} of the first author on the spectral distribution of the twisted Laplacian on typical hyperbolic surfaces with large genus.

\subsection{Related works on resonances for convex co-compact hyperbolic surfaces}
\label{s:cocompact}
The spectral distribution for the twisted Laplacian on a compact hyperbolic surfaces resembles the resonance distribution for a convex co-compact hyperbolic surfaces in many ways. We refer to the book by Borthwick \cite{borthwick} for a fairly complete overview on the later subject. In particular, our work is very much motivated by the early work of Jakobson--Naud \cite{jakobsonnaud} on the essential spectral gap for resonances (as well as similar related results on Pollicott--Ruelle resonances for Anosov flows by Jin--Zworski \cite{localtrace}). For a convex co-compact Fuchsian subgroup $\Gamma$ of $\PSL(2,\mathbb{R})$, the corresponding hyperbolic surface $X=\mathbb{H}/\Gamma$ has infinite volume and one can define the scattering resonance $\Res(-\Delta_X)\subset\mathbb{C}$ for the Laplacian operator on $X$ as the poles of the meromorphic continuation of the resolvent $\left(-\Delta_X-z^2-\frac{1}{4}\right)^{-1}:C_c^\infty(X)\to C^\infty(X)$. This again corresponds to the zeroes of the Selberg zeta function $Z_X$ for this convex co-compact hyperbolic surface. Let $\delta\in(0,1)$ be the Hausdorff dimension of the limit set of $\Gamma$, then the highest resonance is $z_0=i(\delta-\frac{1}{2})$ and the essential spectral gap can be defined as
$$G_X:=\inf\{A\in\mathbb{R}: \#\Res(-\Delta_X)\cap\{\Im z>A\}<\infty\}.$$
Here we always have $G_X\leq 0$ as there are only finitely many resonances in $\{\Im z\geq0\}$. 
\begin{itemize}
\item For the upper bound on $G_X$, there is a natural pressure bound $G_X\leq \delta-\frac{1}{2}$ which is better than the trivial bound $G_X\leq0$ when $\delta<\frac{1}{2}$. Naud \cite{naud} improves to $G_X<\delta-\frac{1}{2}$ and Dyatlov--Jin \cite{regfup} gives a quantitative version. On the other hand, Bourgain--Dyatlov \cite{fullgap} improved the trivial bound $G_X<0$ with the quantitative version given by Jin--Zhang \cite{effup}.
\item The first lower bound on $G_X$ as well as a weak version of sublinear growth is proved by Guillope--Zworski \cite{guillopezw}. The sublinear growth similar to Theorem \ref{t:sublinear} is implicitly contained in Jin--Tao \cite{axioma}. Jakobson--Naud \cite{jakobsonnaud} improve to $G_X\geq-\frac{1}{2}(1-\delta+2\delta^2)$ and conjecture that $G_X=\frac{\delta-1}{2}$. Note the slight different convention here comparing to \cite{jakobsonnaud}.
\end{itemize}

\subsection{Organization of the paper}
\label{s:organ}
The paper is organized as follows. In Section \ref{s:selberg}, we review some basic facts about the spectrum of the twisted Laplacian, including the twisted Selberg trace formula. In Section \ref{s:dynamic}, we review some basic concepts from hyperbolic dynamical systems adapting to the geodesic flow on a compact hyperbolic surface $X$. In Section \ref{s:proof}, we give the proof for the main theorem. In Section \ref{s:nonque}, we explain how to show quantum unique ergodicity fails when the essential spectral gap is positive, following the idea of Anantharaman \cite{nalinideviation}.

\subsection*{Notation}
We use the following notation in the paper: The constant $C>0$ in the inequalities will vary from place to place, depending on the surface $X$, the harmonic 1-form $\omega$ and the test function chosen in the proof, but not other parameters if not specified.

\subsection*{Acknowledgement}
The work is supported by the National Key R \& D Program of China 2022YFA100740. We would like to thank Fr\'{e}d\'{e}ric Naud to suggestion for the paper \cite{jakobsonnaud}, and thank Nalini Anantharaman and Laura Monk for numerous useful discussions. 

\section{Preliminaries}
  \label{s:prelims}
  
\subsection{The twisted Laplacian and Selberg trace formula}
\label{s:selberg}
We can equivalently define the twisted Laplacian operator \eqref{e:twisted-laplacian} as follows: Fix a point $o\in\mathbb{H}$ and lift the harmonic form $\omega$ from $X=\mathbb{H}/\Gamma$ to $\mathbb{H}$. For any $f\in C^\infty(\mathbb{H})$,
$$\Delta_\omega f:=e^{\int_0^x\omega}\Delta(e^{-\int_0^x\omega}f).$$
This coincides with \eqref{e:twisted-laplacian} when $f$ is $\Gamma$-invariant. In this way, the twisted Laplacian $\Delta_{\omega}$ has the same spectrum as the twisted Bochner Laplacian $\Delta_{\rho}$ with the one dimensional representation $\rho: \Gamma=\pi_1(X) \to \mathbb{C}$ defined as 
$$\rho(\gamma):=e^{\int_{\gamma}\omega}.$$
For any representation $\rho:\pi_1(X)=\Gamma\to\GL(V)$, Naud--Spilioti \cite{naudspilioti} define the critical exponent $\delta(\rho)$ of $\Delta_\rho$ and show that
$\Spec(-\Delta_\rho)\subset\mathcal{C}_{\delta(\rho)}$,
where $\mathcal{C}_\sigma,\sigma>\frac{1}{2}$ is the following parabolic region
$$\mathcal{C}_\sigma:=\left\{\Re\lambda\geq\sigma(1-\sigma)+\frac{(\Im\lambda)^2}{(1-2\sigma)^2}\right\}.$$
Passing to the case of one-dimensional representations \eqref{e:twisted-laplacian} and the spectral parameter $r$ with $\lambda=\frac{1}{4}+r^2$ as in the beginning of the paper, we have
$$\Sp(\omega)\subset\left\{0\leq\Im r\leq\delta(\omega)-\frac{1}{2}\right\},$$
and $r_0=i\left(\delta(\omega)-\frac{1}{2}\right)$ where the critical exponent reads
\begin{equation}
\label{e:critical-exponent}
\delta(\omega):=\inf \left\{s>0:\sum_{\gamma\in\mathcal{G}(X)}e^{\int_{\gamma}\omega-s\ell_\gamma}<\infty\right\}.
\end{equation}
Here we denote $\mathcal{G}(X)$ as the collection of all oriented closed geodesics $\ell$ on $X$. For any $\ell\in\mathcal{G}(X)$, we use $\ell_\gamma$ to denote its length and $\ell_\gamma^\#$ its primitive length.

Anantharaman \cite{nalinideviation} (see also M\"{u}ller \cite{muller}) proved the following twisted Selberg trace formula relating the spectrum of $\Delta_\omega$ with the geodesics on $X$, analogous to the usual Selberg trace formula \cite{selberg}: For any even functions $g=g(s):\mathbb{R}\to\mathbb{C}$ which is smooth and decays faster enough, for example, $g\in C_c^\infty(\mathbb{R})$,
\begin{equation}
\label{e:twisted-selberg}
\sum_{j=0}^\infty\hat{g}(r_j)=\frac{\Vol(X)}{4\pi}\int_{-\infty}^\infty r\hat{g}(r)\tanh(\pi r)dr
+\sum_{\gamma\in\mathcal{G}(X)}\frac{e^{\int_\gamma\omega}\ell_{\gamma}^\# g(\ell_\gamma)}{2\sinh(\ell_\gamma/2)}.
\end{equation}
Here $\hat{g}:\mathbb{R}\to\mathbb{C}$ is the Fourier(-Laplace) transform of $g$,
$$\hat{g}(r)=\int_{\mathbb{R}}e^{-irs}g(s)ds.$$
By Paley--Wiener--Schwarz theorem (see H\"{o}rmander \cite[\S 7.3]{hormander}), $\hat{g}$ is an entire function on $\mathbb{C}$ and if $\supp g\subset[-T,T]$, then for any $M\in\mathbb{N}$, there exists $C=C_{M,g}$ such that
\begin{equation}
\label{e:pws}
|\hat{g}(r)|\leq Ce^{T|\Im r|}(1+|\Re r|)^{-M},\quad r\in\mathbb{C}.
\end{equation}
Therefore the left-hand side and the first term of the right-hand side of \eqref{e:twisted-selberg} converges absolutely.

\subsection{Dynamical preliminaries}
\label{s:dynamic}
In this subsection, we review the thermodynamic formalism (see e.g. \cite{bowen, ruelle}) in the setting of the geodesic flow $\varphi^t$ on a compact hyperbolic surface $X$. Let $\mathcal{M}$ be the space of $\varphi^{t}$-invariant probability measure on $S^\ast X$. For any $\mu \in \mathcal{M}$, we denote by $h_{\KS}(\mu)$ its Kolmogorov--Sinai entropy, which is an affine function of $\mu$. Moreover, for the geodesic flow on compact hyperbolic surface, we have $0\leq h_{\KS}(\mu)\leq 1$ and
\begin{itemize}
\item for a $\delta$-measure $\delta_\gamma$ supported on a closed orbit for $\varphi^t$, $h_{\KS}(\delta_\gamma)=0$;
\item $h_{\KS}(\mu)=1$ if and only if $\mu$ is the Liouville measure.
\end{itemize}
The pressure $\Press: C^0(S^\ast X;\mathbb{R})\to\mathbb{R}$ is then defined as the Legendre transform of the Kolmogorov--Sinai entropy function $h_{\KS}:\mathcal{M}\to\mathbb{R}$:
\begin{equation}
\label{e:pressure}
\Press(f)=\sup_{\mu \in \mathcal{M}}\left\{h_{\KS}(\mu)+\int_{S^\ast X}fd\mu\right\},\quad f\in C^0(S^\ast X)
\end{equation}

If $f$ is H\"{o}lder, then the supremum is attained for a unique $\mu$, called the equilibrium measure of $f$. The functional $\Press$ is analytic on H\"{o}lder space. 

Now, we identify the harmonic 1-form $\omega\in\mathcal{H}^1(X;\mathbb{R})$ as the function $\omega(x,\xi):=\langle \omega,\xi\rangle_{x}$ on $T^\ast X$ or its restriction on $S^\ast X$. Then we have the following equilibrium distribution Theorem, see Kifer \cite{kifer}:
\begin{equation}
\label{e:equilibrium}
\lim_{t\to \infty}\frac{1}{t}\log \left(\sum\limits_{\gamma \in \mathcal{G}(X), \ |\ell_{\gamma}-t|\leq \frac{1}{2}}e^{\int_{\gamma}\omega}\right)=\Press(\omega).
\end{equation}
In particular, we have the following asymptotic growth, see Parry--Pollicott \cite{pp}, 
$$\sum\limits_{\gamma \in \mathcal{G}(X), \ |\ell_{\gamma}-t|\leq \frac{1}{2}}e^{\int_{\gamma}\omega} \sim \frac{e^{t\mathrm{Pr}(\omega)}}{t\Press(\omega)},$$ 
This implies that the critical exponent \eqref{e:critical-exponent} is exactly the pressure:
\begin{equation}
\label{e:critical-pressure}
\delta(\omega)=\Press(\omega).
\end{equation}

Next, we introduce the stable norm of the $\omega \in \mathcal{H}^{1}(X,\mathbb{R})$:
\begin{equation}
\label{e:stablenorm}
\|\omega\|_{s}=\sup_{\mu \in \mathcal{M}}\int_{S^*X}\omega d\mu=\sup_{\gamma \in \mathcal{G}(X)}\bar{\omega}_\gamma.
\end{equation}
Here we define $\bar{\omega}_\gamma$ is the average of $\omega$ along $\gamma$:
\begin{equation}
\label{e:omega-average}
\bar{\omega}_\gamma=\frac{\int_\gamma\omega}{\ell_\gamma}.
\end{equation}
In particular, since $h_{\KS}(\mu)=1$ for any $\mu\in\mathcal{M}$,
\begin{equation}
\label{ineqpress}
\max\{1,\|\omega\|_{s}\}\leq \mathrm{Pr}(\omega) \leq 1+\|\omega\|_{s}.
\end{equation}

Besides, the restriction of $\Press$ to any line $\{f+tg, t \in \mathbb{R}\}\subset C^\infty(T^\ast X;\mathbb{R})$ is strictly convex, unless $g$ is cohomologous to a constant \cite{ratner}, i.e. there exists a function $h\in C^\infty(T^\ast X;\mathbb{R})$ such that $g=\bar{g}+Vh$ where $V$ is the generating vector field of the geodesic flow $\varphi^t$ and $\bar{g}=\int_{S^\ast X}gd\mu_L$ is the average of $g$ in $S^\ast M$ under the probabilistic Liouville measure. An equivalent description for $g$ cohomologous to a constant is that for all unit speed closed geodesics $\gamma(t)$,
$$\int_{0}^{\ell_{\gamma}} g(\gamma(t))dt =\bar{g}\ell_\gamma.$$
Since any nonzero $\omega\in\mathcal{H}^1(X;\mathbb{R})$ is not cohomologous to a constant. We have $\Press(\omega)=1$ if and only of $\omega=0$. 

\section{Proof of the main theorems}
  \label{s:proof}
  
\subsection{Sublinear growth}
\label{s:sublinear}
We first prove Theorem \ref{t:sublinear} following the strategy in \cite{axioma} which originates from \cite{guillopezw} and \cite{localtrace} but with some improvement. 

We fix a function $\varphi\in C_c^\infty(\mathbb{R})$ with
the following properties: $\varphi(s)\geq0$ for all $s\in\mathbb{R}$, $\varphi(0)=1$ and $\supp\varphi\subset(-1,1)$. For any $0<\varepsilon<1<d$, we rescale the test function $\varphi$ to 
$$\varphi_{\varepsilon,d}(s):=\varphi\left(\frac{s-d}{\varepsilon}\right)\geq0$$
so that $\varphi_{\varepsilon,d}(d)=1$ and 
$$\supp\varphi_{\varepsilon,d}\subset(d-\varepsilon,d+\varepsilon)\subset(0,\infty).$$
The Paley--Wiener--Schwarz theorem \eqref{e:pws} shows that there is a constant depending only on $M>0$ and $\varphi$ such that for any $r\in\mathbb{C}$,
$$|\widehat{\varphi_{\varepsilon,d}}(r)
=|\varepsilon\hat{\varphi}(\varepsilon r)e^{-idr}|
\leq C_M\varepsilon(1+\varepsilon|\Re r|)^{-M}e^{(d+\varepsilon)|\Im r|}.$$
We choose the test function $g$ in the twisted Selberg trace formula \eqref{e:twisted-selberg} to be 
$$g(s)=\varphi_{\varepsilon,d}(s)+\varphi_{\varepsilon,d}(-s)$$
then for any $M>0$, there exists a constant $C_M$ such that for any $r\in\mathbb{C}$ with $\Im r\geq0$ and $0<\varepsilon<1<d$,
\begin{equation}
\label{e:pws-scale}
|\hat{g}(r)|=|\widehat{\varphi_{\varepsilon,d}}(r)+\widehat{\varphi_{\varepsilon,d}}(-r)|
\leq C_M\varepsilon(1+\varepsilon|\Re r|)^{-M}e^{(d+\varepsilon)\Im r}.
\end{equation}
Let us estimate all terms in \eqref{e:twisted-selberg} as follows: The sum on the left-hand side is again separated into two parts 
$$\sum_{r_j\in\Sp(X,\omega),\Im r_j<A}\hat{g}(r_j)+\sum_{r_j\in\Sp(X,\omega),\Im r_j\geq A}\hat{g}(r_j).$$
\begin{itemize}
\item In the first sum, we use \eqref{e:psi-scale} with $0\leq\Im r_j<A$, $M=3$, to get
$$\left|\sum_{\Im r_j<A}\hat{g}(r_j)\right|
\leq C\varepsilon e^{(d+\varepsilon)A}\sum_{\Im r_j<A}(1+\varepsilon|\Re r|)^{-3}
\leq C\varepsilon e^{(d+\varepsilon)A}\int_0^\infty(1+\varepsilon R)^{-3}dN_0(R).$$
By Weyl law \eqref{e:weyllaw}, there exists $C=C(X,\omega)>0$ such that $N_0(R)\leq C(1+R^2)$ for any $R\geq 0$ and thus with a constant $C$ only depending on $X,\omega$ and $\varphi$,
\begin{equation}
\label{e:sublinear-leftsmall}
\left|\sum_{\Im r_j<A}\hat{g}(r_j)\right|
\leq C\varepsilon e^{(d+\varepsilon)A}\int_0^\infty(1+\varepsilon R)^{-3}dN_0(R)\leq C\varepsilon^{-1}e^{(d+\varepsilon)A}.
\end{equation}
\item In the second sum, we use $N_A(R)$ in a similar way to get for any $M>0$, there exists a constant $C_M>0$ such that
$$\left|\sum_{\Im r_j>A}\hat{g}(r_j)\right|\leq C_M\varepsilon e^{(d+\varepsilon)(\Press(\omega)-\frac{1}{2})}\int_0^\infty(1+\varepsilon R)^{-M}dN_A(R).$$
Integration by parts and use the change of variables $t=\varepsilon R$ to get
$$\left|\sum_{\Im r_j>A}\hat{g}(r_j)\right|\leq C_M\varepsilon e^{(d+\varepsilon)(\Press(\omega)-\frac{1}{2})}\int_0^\infty(1+t)^{-M-1}N_A(\varepsilon^{-1}t)dt.$$
Here we also absorb the extra term $N_A(0)$ into the integral. Now we further separate the integral into two parts: Fix $a>0$ chosen later,
$$\int_0^\infty(1+t)^{-M-1}N_A(\varepsilon^{-1}t)dt
=\int_0^{\varepsilon^{-a}}+\int_{\varepsilon^{-a}}^\infty(1+t)^{-M-1}N_A(t/\varepsilon)dt.$$
For any $M>2$, we estimate the first integral by
\begin{equation}
\label{e:sublinear-leftlarge1}
\int_0^{\varepsilon^{-a}}(1+t)^{-M-1}N_A(\varepsilon^{-1}t)dt\leq C_MN_A(\varepsilon^{-a-1}).
\end{equation}
and the second integral similar to \eqref{e:sublinear-leftsmall} with $N_A(\varepsilon^{-1}t)\leq  N_0(\varepsilon^{-1}t)\leq C\varepsilon^{-2}t^2$ to get
\begin{equation}
\label{e:sublinear-leftlarge2}
\int_{\varepsilon^{-a}}^\infty(1+t)^{-M-1}N_A(\varepsilon^{-1}t)dt
\leq \varepsilon^{-2}\int_{\varepsilon^{-a}}^\infty(1+t)^{-M+1}dt\leq C_M\varepsilon^{a(M-2)-2}.
\end{equation}
Combining \eqref{e:sublinear-leftlarge1} and \eqref{e:sublinear-leftlarge2}, there exists a constant $C_M$ only depending on $X,\omega,\varphi$ and $M$ such that for any $0<\varepsilon<1<d$ and $a>0$,
\begin{equation}
\label{e:sublinear-leftlarge}
\left|\sum_{\Im r_j>A}\hat{g}(r_j)\right|\leq C_M\varepsilon e^{d(\Press(\omega)-\frac{1}{2})}(N_A(\varepsilon^{-a-1})+\varepsilon^{a(M-2)-2}).
\end{equation}
Here we further absorb the factor $e^{\varepsilon(\Press(\omega)-\frac{1}{2})}$ to the constant $C_M$ as $\varepsilon<1$.
\end{itemize}
The first term on the right-hand side is again estimated by \eqref{e:pws-scale} with $M=3$:
\begin{equation}
\label{e:sublinear-right1}
\frac{\Vol(X)}{4\pi}\left|\int_{-\infty}^\infty r\hat{g}(r)\tanh(\pi r)dr\right|
\leq C\int_{\mathbb{R}}\varepsilon|r|(1+\varepsilon|r|)^{-3}dr\leq C\varepsilon^{-1}.
\end{equation}
For the second term, we now choose $d=k\ell_0$ where $k\in\mathbb{N}^\ast$ and $\ell_0=\ell_{\gamma_0}$ where $\gamma_0\in\mathcal{P}(X)$ is chosen later. Since $g\geq0$, we only keep the term with $\gamma=k\gamma_0$ so that  $g(\ell_\gamma)=1$ and $\ell_\gamma^\#=\ell_0$ to get
\begin{equation}
\label{e:sublinear-right2}
\sum_{\gamma\in\mathcal{G}(X)}\frac{e^{\int_\gamma\omega}\ell_{\gamma}^\# g(\ell_\gamma)}{2\sinh(\ell_\gamma/2)}\geq
\frac{e^{k\int_{\gamma_0}\omega}\ell_0}{2\sinh(k\ell_0/2)}
\geq \ell_0e^{k\ell_0(\bar{\omega}_{\gamma_0}-\frac{1}{2})}.
\end{equation}
Here we recall $\bar{\omega}_{\gamma_0}$ is defined in \eqref{e:omega-average} as the average of $\omega$ along $\gamma_0$. 

Now we combine all the estimates \eqref{e:sublinear-leftsmall}, \eqref{e:sublinear-leftlarge}, \eqref{e:sublinear-right1} and \eqref{e:sublinear-right2}, so that there exists a constant $C$ only depending on $X,\omega,\varphi$ and a constant $C_M$ which can further depend on $M>2$ such that for any $\varepsilon\in(0,1)$, $a>0$, $A\in(0,\Press(\omega)-\frac{1}{2})$, $\gamma_0\in\mathcal{P}(X)$, $k\in\mathbb{N}$ sufficiently large so that $d=k\ell_0>1$, 
$$\ell_0e^{k\ell_0(\bar{\omega}_{\gamma_0}-\frac{1}{2})}
\leq C_M\varepsilon e^{k\ell_0(\Press(\omega)-\frac{1}{2})}(N_A(\varepsilon^{-a-1})+\varepsilon^{a(M-2)-2})+C\varepsilon^{-1}e^{k\ell_0A},$$
which is
$$\ell_0\varepsilon^{-1}e^{k\ell_0(\bar{\omega}_{\gamma_0}-\Press(\omega))}\leq C_M(N_A(\varepsilon^{-a-1})+\varepsilon^{a(M-2)-2})+C\varepsilon^{-2}e^{-k\ell_0(\Press(\omega)-\frac{1}{2}-A)}.$$
Now we let $\varepsilon=e^{-d/b}=e^{-k\ell_0/b}$ where $b>0$  to be chosen later and get
$$N_A(\varepsilon^{-a-1})\geq C_M^{-1}\ell_0\varepsilon^{-1-b(\bar{\omega}_{\gamma_0}-\Press(\omega))}-C_M\varepsilon^{-2+b(\Press(\omega)-\frac{1}{2}-A)}-\varepsilon^{a(M-2)-2}.$$
Now for any $\beta\in(0,1)$ and $a\in(0,\beta^{-1}-1)$, we take 
$$b=\frac{1-\beta(a+1)}{\Press(\omega)-\bar{\omega}_{\gamma_0}}>0$$
and $M>2+(2-\beta(a+1))/a$ so that
$$2-a(M-2)<\beta(a+1)=1+b(\bar{\omega}_{\gamma_0}-\Press(\omega)).$$
If $A>0$ satisfies
$$2-b(\Press(\omega)-\frac{1}{2}-A)<1+b(\bar{\omega}_{\gamma_0}-\Press(\omega))$$
which is just
\begin{equation}
\label{e:range-A-rough}
A<\bar{\omega}_{\gamma_0}-\frac{1}{2}-\frac{1}{b}=
\bar{\omega}_{\gamma_0}-\frac{1}{2}-\frac{\Press(\omega)-\bar{\omega}_{\gamma_0}}{1-\beta(a+1)},
\end{equation}
we get for any $k\in\mathbb{N}_\ast$ large enough, $\varepsilon=e^{-k\ell_0/b}$, there exists a constant $C>0$ depending on $X,\omega$ as well as $\varphi,M,A,a$,
$$N_A(\varepsilon^{-a-1})\geq\frac{1}{C}\varepsilon^{-\beta(a+1)}.$$
This implies that for any $\beta'<\beta$, there exists $R_0>0$ such that for any $R>R_0$, if $R\in(e^{k\ell_0(a+1)/b},e^{(k+1)\ell_0(a+1)/b})$, and then
$$N_A(R)\geq N_A(e^{k\ell_0(a+1)/b})\geq\frac{1}{C}
e^{\beta k\ell_0(a+1)b}\geq\frac{1}{C}R^{\beta k/(k+1)}\geq 
\frac{1}{C}R^{\beta'}.$$
Now in \eqref{e:range-A-rough}, we can choose $a$ arbitrarily small and by \eqref{e:stablenorm}, we take supreme over all $\gamma_0\in\mathcal{P}(X)$ to get $\|\omega\|_s$ instead of $\bar{\omega}_{\gamma_0}$ to get as long as \eqref{e:range-A}, we get the sublinear growth \eqref{e:sublinear} of $N_A(R)$ and this finishes the proof of Theorem \ref{t:sublinear}.

\begin{rem}\label{generalizetohigherdimension}
Our proof of Theorem \ref{t:sublinear} can be applied to 
$\Delta_\rho$ with representation $\rho:\Gamma\to\GL(V)$ satisfying the property $\tr(\rho(\gamma))\geq0$ for any $\gamma\in\Gamma$. However, in general, $\tr(\rho(\gamma)^k)\neq\tr(\rho(\gamma))^k$, we can only get the following lower bound of the essential spectral gap $G_\rho$ defined similarly as \eqref{e:essgap}
$$G_\rho\geq 2s(\rho)-\delta(\rho)-\frac{1}{2},$$
where $\delta(\rho)$ is the critical exponent  (see \cite{naudspilioti}) and 
$$s(\rho):=\limsup_{\ell_\gamma\to +\infty}\frac{\log\tr(\rho(\gamma))}{\ell_\gamma}.$$
Note that we still have $s(\rho)\leq\delta(\rho)\leq1+s(\rho)$ and thus $G_\rho>0$ if $s(\rho)>\frac{3}{2}$.
\end{rem}

\subsection{Essential spectral gap}
\label{s:essgap}
In this section, we prove theorem \ref{t:essgap-lb}. Our method follows from Jakobson--Naud \cite{jakobsonnaud}, see also the appendix of Jin--Zworski \cite{localtrace} by Naud.

We choose an even function $\psi\in C_c^\infty((-1,1))$ such that $\psi\geq0$ everywhere and $\psi(s)=1$ for $|s|\leq\frac{1}{2}$. For $t\geq 1$ and $\xi\in\mathbb{R}$, we define
\begin{equation}
\label{e:psi-scale}
\psi_{t,\xi}(s):=e^{is\xi}\psi(s-t),
\end{equation}
and we shall apply the twisted Selberg trace formula \eqref{e:twisted-selberg} to the following even function
\begin{equation}
\label{e:test}
g(s):=\psi_{t,\xi}(s)+\psi_{t,\xi}(-s)\in C_c^\infty(\mathbb{R}).
\end{equation}
For $t\geq 1$, $g(\ell_\gamma)=e^{i\ell_\gamma\xi}\psi(\ell_\gamma-t)$ and thus the second term in the right-hand side of \eqref{e:twisted-selberg} becomes
\begin{equation}
\label{e:essgap-S}
S(t,\xi):=\sum_{\gamma\in\mathcal{G}(X)}\frac{e^{\int_\gamma\omega+i\ell_\gamma\xi}\ell_\gamma^\#\psi(\ell_\gamma-t)}{2\sinh(\ell_\gamma/2)}
\end{equation}
The main idea is to estimate the average of $|S(t,\xi)|^2$ against a Gaussian weight:
\begin{equation}
\label{e:gaussianaverage}
I(t,\sigma):=\frac{1}{\sqrt{2\pi}\sigma}\int_{\mathbb{R}}|S(t, \xi)|^2 e^{-\xi^2/2\sigma^2} d\xi.
\end{equation}
A direct calculation using
$$\frac{1}{\sqrt{2\pi}\sigma}\int_{\mathbb{R}}e^{i\ell_\gamma\xi-i\ell_{\tilde{\gamma}}\xi}e^{-\xi^2/2\sigma^2}d\xi
=e^{-\frac{1}{2}\sigma^2(\ell_\gamma-\ell_{\tilde{\gamma}})^2}$$
shows that
$$I(t,\sigma)=
\sum_{\gamma,\tilde{\gamma}\in\mathcal{G}(X)}
\frac{\ell_\gamma^\#\ell_{\tilde{\gamma}}^\#e^{-\frac{1}{2}\sigma^2(\ell_\gamma-\ell_{\tilde{\gamma}})^2+\int_{\gamma}\omega+\int_{\tilde{\gamma}}\omega}\psi(\ell_\gamma-t)\psi(\ell_{\tilde{\gamma}}-t)}{4\sinh(\ell_\gamma/2)\sinh(\ell_{\tilde{\gamma}}/2)}.$$
Noticing that all terms are nonnegative, we only keep the diagonal terms $\gamma=\tilde{\gamma}$ and restrict to the terms with $\psi(\ell_\gamma-t)=1$ leading to the following lower bound for \eqref{e:gaussianaverage}:
$$I(t,\sigma)\geq\sum_{|\ell_\gamma-t|\leq\frac{1}{2}}\frac{(\ell_\gamma^\#)^2e^{2\int_\gamma\omega}}{4\sinh^2(\ell_\gamma/2)}.$$
Applying the equilibrium distribution theorem \eqref{e:equilibrium} gives for any $\varepsilon>0$, there exists $c>0$ such that for all $t\geq 1$ and $\sigma>0$,
\begin{equation}
\label{e:gaussianlb}
I(t,\sigma)\geq ce^{(\Press(2\omega)-1-\varepsilon)t}.
\end{equation}

On the other hand, by the Paley--Wiener--Schwarz theorem (see \cite{hormander}), 
$$\widehat{\psi_{t,\xi}}(r)=e^{it(\xi-r)}\hat{\psi}(r-\xi)$$ is an entire function of $r\in\mathbb{C}$ with the estimate
$$\widehat{\psi_{t,\xi}}(r)\leq C_Me^{(t+1)|\Im r|}(1+|\Re r-\xi|)^{-M},$$
for any $M>0$. This gives the estimate for the entire function $\hat{g}(r)$ uniformly in $t\geq 1$ and $\xi\in\mathbb{R}$:
\begin{equation}
\label{e:pws-essgap}
|\hat{g}(r)|\leq C_Me^{(t+1)|\Im r|}((1+|\Re r-\xi|)^{-M}+(1+|\Re r+\xi|)^{-M}).
\end{equation}

Now we argue by contradiction, and assume that $N_A(R)=\mathcal{O}(1)$, i.e. there are only finitely many $r_j$ with $\Im r_j\geq A$. We obtain upper bound for \eqref{e:essgap-S} through the twisted Selberg trace formula \eqref{e:twisted-selberg}:
$$S(t,\xi)=\sum_{j=0}^\infty\hat{g}(r_j)-\frac{\Vol(X)}{4\pi}\int_{-\infty}^\infty r\hat{g}(r)\tanh(\pi r)dr.$$
Therefore 
\begin{equation}
\label{e:essgap-S-ub}
|S(t,\xi)|^2\leq C\left|\sum_{\Im r_j<A}\hat{g}(r_j)\right|^2+C\left|\sum_{\Im r_j\geq A}\hat{g}(r_j)\right|^2+C\left|\int_{-\infty}^\infty r\hat{g}(r)\tanh(\pi r)dr\right|^2.
\end{equation}
We use \eqref{e:pws-essgap} to estimate each term on the right-hand side of \eqref{e:essgap-S-ub}: For $t\geq 1$,
\begin{itemize}
\item In the first term in \eqref{e:essgap-S-ub}, by the Weyl law \eqref{e:weyllaw}, we can take $M=3$ and get
$$\sum_{j=0}^\infty(1+|\Re r_j-\xi|)^{-3}\leq C(1+|\xi|).$$
and thus
\begin{equation}
\label{e:essgap-S1}
\left|\sum_{\Im r_j<A}\hat{g}(r_j)\right|\leq C(1+|\xi|)e^{tA}.
\end{equation}
\item Since for $\Im r_j\geq A$, $\Re r_j$ is bounded, each term in the finite sum in the second term in \eqref{e:essgap-S-ub} can be estimated by
\begin{equation}
\label{e:essgap-S2}
|\hat{g}(r_j)|\leq C_Me^{t(\Press(\omega)-\frac{1}{2})}
(1+|\xi|)^{-M}.
\end{equation}
\item For the last term in \eqref{e:essgap-S-ub}, we simply take $M=3$ in \eqref{e:pws-essgap} to get
\begin{equation}
\label{e:essgap-S3}
\left|\int_{-\infty}^\infty r\hat{g}(r)\tanh(\pi r)dr\right|
\leq C(1+|\xi|).
\end{equation} 
\end{itemize}
Now combining \eqref{e:essgap-S1}, \eqref{e:essgap-S2} and \eqref{e:essgap-S3} with the following estimates on Gaussian average: for any $\sigma\geq 1$,
$$\frac{1}{\sqrt{2\pi}\sigma}\int_\mathbb{R}(1+|\xi|)^2e^{-|\xi|^2/2\sigma^2}d\xi\leq C\sigma^2.$$
and 
$$\frac{1}{\sqrt{2\pi}\sigma}\int_\mathbb{R}(1+|\xi|)^{-2M}e^{-|\xi|^2/2\sigma^2}d\xi\leq C_M\sigma^{-1},$$
we obtain the upper bound for \eqref{e:gaussianaverage}: For $t,\sigma\geq 1$,
\begin{equation}
\label{e:gaussianub}
I(t,\sigma)\leq Ce^{2tA}\sigma^2+C_Me^{t(2\Press(\omega)-1)}\sigma^{-1}.
\end{equation}
Now comparing \eqref{e:gaussianlb} with \eqref{e:gaussianub} and taking $\sigma=e^{bt}$, $b>0$, we have for every $\varepsilon>0$, there exists a constant $C>0$ depending also on $X$, $\omega$, $A$ and $\psi$ such that for any $t\geq 1$ and $b>0$,
$$e^{(\Press(2\omega)-1-\varepsilon)t}
\leq C(e^{2(A+b)t}+e^{(2\Press(\omega)-1-b)t}).$$
Thus we have a contradiction when $t\to+\infty$ if
$$A<\frac{1}{2}(\Press(2\omega)-1)-(2\Press(\omega)-\Press(2\omega))$$
and we choose $\varepsilon>0$ small enough,
$$b=2\Press(\omega)-\Press(2\omega)+2\varepsilon.$$
This finishes the proof of Theorem \ref{t:essgap-lb}.

We can further elaborate the analysis if we do not assume $N_A(R)=\mathcal{O}(1)$ as follows: To estimate the contribution of the second term on the right-hand side of \eqref{e:essgap-S-ub} to $I(t,\sigma)$, we use \eqref{e:pws-essgap} to write 
$$\tilde{S}(t,\xi):=\left|\sum_{\Im r_j\geq A}\hat{g}(r_j)\right|
\leq C_Me^{t(\Press(\omega)-\frac{1}{2})}\int_0^\infty
(1+|R-\xi|)^{-M}+(1+|R+\xi|)^{-M}dN_A(R).$$
Then by the change of variable $\xi=\sigma\eta$ and the Minkowski inequality,
\begin{equation*}
\begin{split}
\frac{1}{\sqrt{2\pi}\sigma}&\;\int_\mathbb{R}|\tilde{S}(t,\xi)|^2e^{-|\xi|^2/2\sigma^2}d\xi=\frac{1}{\sqrt{2\pi}}\int_{\mathbb{R}}|\tilde{S}(t,\sigma\eta)|^2e^{-\eta^2/2}d\eta\\
\leq&\; C_Me^{t(2\Press(\omega)-1)}\left(\int_0^\infty\left(\int_{\mathbb{R}}(1+|R-\sigma\eta|)^{-2M}e^{-\eta^2/2}d\eta\right)^{1/2}dN_A(R)\right)^2.
\end{split}
\end{equation*}
Thus we need estimate for the integral
$$J_M(r,\sigma):=\frac{1}{\sqrt{2\pi}\sigma}\int_{\mathbb{R}}(1+|r-\xi|)^{-M}e^{-\xi^2/2\sigma^2}d\xi
=\frac{1}{\sqrt{2\pi}}\int_{\mathbb{R}}(1+|r-\sigma\eta|)^{-M}e^{-\eta^2/2}d\eta,$$
uniformly in $r>0$ and $\sigma\geq 1$. We divide the integral into two parts and estimate separately:
$$\int_{|\eta-\frac{r}{\sigma}|\leq\frac{r}{2\sigma}}(1+|r-\sigma\eta|)^{-M}e^{-\eta^2/2}d\eta
\leq e^{-r^2/8\sigma^2}\int_\mathbb{R}(1+|r-\sigma\eta|)^{-M}d\eta\leq C_M\sigma^{-1}e^{-r^2/8\sigma^2};$$
$$\int_{|\eta-\frac{r}{\sigma}|\geq\frac{r}{2\sigma}}(1+|r-\sigma\eta|)^{-M}e^{-\eta^2/2}d\eta
\leq 2\int_{r/2\sigma}^\infty(1+\sigma\eta)^{-M}d\eta\leq C_M\sigma^{-1}(1+r)^{1-M}.$$
Therefore we have for $M\geq 2$,
$$J_M(r,\sigma)\leq C_M\sigma^{-1}(e^{-r^2/8\sigma^2}+(1+r)^{1-M}),$$
and thus using 
$$\int_0^\infty J_{2M}(R,\sigma)^{1/2}dN_A(R)
\leq C_M\sigma^{-1/2}\int_0^\infty \left(e^{-R^2/8\sigma^2}+(1+R)^{\frac{1}{2}-M}\right)dN_A(R)$$
The second term in this integral can be bounded by a constant by the Weyl law \eqref{e:weyllaw} if we choose $M\geq3$. For the first term we integrate by parts and change variables $R=\sigma u$,
\begin{equation*}
\begin{split}
    \int_0^\infty e^{-R^2/8\sigma^2}dN_A(R)
=&N_A(0)+\frac{1}{4\sigma^2}\int_0^\infty e^{-R^2/8\sigma^2}N_A(R)RdR\\
=&N_A(0)+\frac{1}{4}\int_0^\infty e^{-u^2/8}N_A(\sigma u)udu.
\end{split}
\end{equation*}
As before, we fix some $a>0$ small chosen later and separate the integral into 
$$N_A(0)+\int_0^{\sigma^a}e^{-u^2/8}N_A(\sigma u)udu
\leq CN_A(\sigma^{1+a});$$
and again by the Weyl law \eqref{e:weyllaw},
$$\int_{\sigma^a}^\infty e^{-u^2/8}N_A(\sigma u)udu
\leq C\sigma^2\int_{\sigma^a}^\infty e^{-u^2/8}u^3du
\leq C\sigma^2e^{-\sigma^{2a}/16}$$
which is again bounded by a constant. Thus
\begin{equation}
\label{e:Stilde}
\frac{1}{\sqrt{2\pi}\sigma}\int_{\mathbb{R}}|\tilde{S}(t,\xi)|^2e^{-|\xi|^2/2\sigma^2}d\xi
\leq C_Me^{t(2\Press(\omega)-1)}\sigma^{-1}(N_A(\sigma^{1+a})^2+1).
\end{equation}
Now we use \eqref{e:Stilde} instead of \eqref{e:essgap-S2} to get the upper bound
\begin{equation}
\label{e:gaussianub+}
I(t,\sigma)\leq Ce^{2tA}\sigma^2+Ce^{t(2\Press(\omega)-1)}\sigma^{-1}(1+N_A(\sigma^{1+a})^2),
\end{equation}
where the constant depends on $X,\omega,A,\psi$ and $a>0$, but not on $t,\sigma\geq 1$. Again, comparing with \eqref{e:gaussianlb} and choosing $\sigma=e^{bt}$ with $b>0$ determined later, we get the lower bound
$$N_A(e^{(1+a)bt})^2
\geq \frac{1}{C}e^{(\Press(2\omega)-2\Press(\omega)+b-\varepsilon)t}
-Ce^{(2A-2\Press(\omega)+1+3b)t}-C.$$
Now for any $\beta\in(0,\frac{1}{2})$, any $a\in(0,\frac{1}{2\beta}-1)$, we take 
$$b=\frac{2\Press(\omega)+\varepsilon-\Press(2\omega)}{1-2\beta(1+a)}>0,$$
to see if 
$$A<\frac{1}{2}(\Press(2\omega)-1-2b-\varepsilon)
=\frac{1}{2}(\Press(2\omega)-1)-\frac{2\Press(\omega)-\Press(2\omega)}{1-2\beta(1+a)}-\left(\frac{1}{2}+\frac{1}{1-2\beta(1+a)}\right)\varepsilon,$$
we get \eqref{e:sublinear} when $R=e^{(1+a)bt}$ is sufficiently large and $\beta\in(0,\frac{1}{2})$. We can now take $a,\varepsilon>0$ arbitrarily small to see \eqref{e:sublinear} holds when 
$$A<\frac{1}{2}(\Press(2\omega)-1)-\frac{2\Press(\omega)-\Press(2\omega)}{1-2\beta},\quad \beta\in\left(0,\frac{1}{2}\right).$$

\subsection{The case of the arithmetic surface}
\label{s:arithmetic}
Now we consider the case of the arithmetic surfaces and give an alternative proof of \eqref{e:essgap-arithmetic}. We recall that for a prime number $p\equiv1\mod{4}$ and $n\in\mathbb{Z}$ not a quadratic residue modulo $p$, the fundamental group $\Gamma=\Gamma(n,p)<\PSL(2,\mathbb{R})$ of an arithmetic surface $X=\mathbb{H}^2/\Gamma$ arising from a quanternion algebra consists of all matrices of the form
$$\begin{pmatrix}
a+b\sqrt{n} & (c+d\sqrt{n})\sqrt{p}\\
(c-d\sqrt{n})\sqrt{p} & a-b\sqrt{n}
\end{pmatrix},\quad a,b,c,d\in\mathbb{Z},a^2-b^2n-c^2p+d^2np=1.$$
The length spectrum of $X$ is given by $\{\log x_m\}_{m=0}^\infty$ where
$$x_m=2m^2-1+2m\sqrt{m^2-1},\quad m\in\mathbb{N}.$$
For $m\in\mathbb{N}$, we define
$$L(m):=\sum_{\gamma\in\mathcal{G}(X),\ell_\gamma=\log x_m}
\frac{\ell_\gamma^\#e^{\int_\gamma\omega}}{2\sinh(\ell_\gamma/2)}=\frac{1}{2\sinh(\frac{1}{2}\log x_m)}\sum_{\gamma\in\mathcal{G}(X),\ell_\gamma=\log x_m}\ell_\gamma^\#e^{\int_\gamma\omega}.$$
Then by the equilibrium distribution theorem, for any $\varepsilon>0$, there exists $c>0$ such that for all $t\geq 1$,
$$\sum_{m\in\mathbb{N}:|\log x_m-t|\leq\frac{1}{2}}
L(m)\geq ce^{(\Press(\omega)-\frac{1}{2}-\varepsilon)t}.$$
On the other hand, 
$$\sum_{m\in\mathbb{N}:|\log x_m-t|\leq\frac{1}{2}}
1\geq Ce^{t/2}.$$
By Cauchy--Schwarz inequality, we obtain for any $\varepsilon>0$, there exists $c>0$ such that
$$\sum_{m\in\mathbb{N}:|\log x_m-t|\leq\frac{1}{2}}
L(m)^2\geq ce^{(2\Press(\omega)-\frac{3}{2}-\varepsilon)t}.$$
Now in the Gaussian average \eqref{e:gaussianaverage}, we keep all the terms with $\ell_\gamma=\ell_{\tilde{\gamma}}\in[t-\frac{1}{2},t+\frac{1}{2}]$ to get a better lower bound:
$$I(t,\sigma)\geq\sum_{m\in\mathbb{N}:|\log x_m-t|\leq\frac{1}{2}}\sum_{\ell_\gamma=\ell_{\tilde{\gamma}}=\log x_m}\frac{\ell_\gamma^\#\ell_{\tilde{\gamma}}^\#e^{\int_\gamma\omega}e^{\int_{\tilde{\gamma}}\omega}}{4\sinh^2(\frac{1}{2}\log x_m)}=\sum_{m\in\mathbb{N}:|\log x_m-t|\leq\frac{1}{2}}L(m)^2.$$
Therefore for any $t,\sigma\geq 1$
\begin{equation}
\label{e:gaussianlb-arithmetic}
I(t,\sigma)\geq ce^{(2\Press(\omega)-\frac{3}{2}-\varepsilon)t}.
\end{equation}
Use \eqref{e:gaussianlb-arithmetic} instead of \eqref{e:gaussianlb} in the argument in Section \ref{s:essgap} with $\sigma=e^{t/2}$ we recover Anantharaman's lower bound \eqref{e:essgap-arithmetic}. 
\begin{rem}
The method in Section \ref{s:sublinear} can also be applied here using the fact that there is at least one $L(m)$ with $|\log x_m-t|<\frac{1}{2}$ and $L(m)\geq ce^{(\Press(\omega)-1-\varepsilon)t}$ replacing the weaker estimate \eqref{e:sublinear-right2}. But this seems only giving a worse lower bound $G_\omega\geq\Press(\omega)-\frac{3}{2}$.
\end{rem}
  
\section{Stable norms and the failure of quantum unique ergodicity}
\label{s:nonque}
  
\subsection{Positive essential spectral gap implies non-QUE}
\label{s:gaptoque}
Now we consider the semiclassical defect measures associated to the eigenfunctions $\phi_j$ of $\Delta_\omega$. We use the notation in \cite[Appendix E.3]{resbook} and take the semiclassical parameter $h_j=|\Re r_j |^{-1}\to0+$. In thiw way, we can rewrite \eqref{e:eigenvalues} as
$$P(h_j)\phi_j=0$$
where
$$P(h_j)=-h_j^2\Delta+2h\langle\omega,hd\bullet\rangle
-h_j^2\left(|\omega|^2+\frac{1}{4}+r_j^2\right).$$
Therefore $P\in\Psi^2_h(X)$ satisfies the following conditions
\begin{itemize}
\item $\sigma_h(P)=p(x,\xi):=|\xi|_x^2-1$;
\item $\Im P=\frac{1}{2i}(P-P^\ast)\in h\Psi^1_h(X)$ and
$$\sigma_h(h^{-1}\Im P)=2(\omega(x,\xi)-\Im r_j).$$
\end{itemize}
Now by \cite[Theorem E.43,E.44]{resbook}, we see if we have a subsequence $\phi_{j_k}$ with a semiclassical defect measure $\mu$ and $\Im r_{j_k}\to\alpha\in[0,\infty)$, then 
\begin{itemize}
\item $\supp \mu\subset S^\ast X=\{(x,\xi)\in T^\ast X: |\xi|_x=1\}$;
\item for any $a\in C_c^\infty(T^\ast X)$,
\begin{equation}
\label{e:defect-invariance}
\int_{T^\ast X}(H_pa+2ba)d\mu=0,\quad b=2(\omega(x,\xi)-\alpha)
\end{equation}
Here $H_p$ is the Hamiltonian vector field of $p$ on $T^\ast M$, which is tangent to $S^\ast X$ and equals to twice the generator of the geodesic flow $\varphi^t$ when restricted to $S^\ast X$.
\end{itemize}
In particular, if we choose $a\in C_c^\infty(T^\ast X)$ which equals 1 near $S^\ast X$, so that $H_pa|_{S^\ast X}=0$, then \eqref{e:defect-invariance} shows that
$$\int_{S^\ast X}\omega(x,\xi)d\mu=\alpha\mu(S^\ast X).$$
Now if $G_\omega>0$, we can find one subsequence of $r_j$ with $\alpha_1=0$ and thus 
$$\int_{S^\ast X}\omega(x,\xi)d\mu_1=0$$
by the concentration of eigenvalues \eqref{e:concentration},  and another subsequence with $\alpha_2=G_\omega$ and
$$\int_{S^\ast X}\omega(x,\xi)d\mu_2=G_\omega\mu_2(S^\ast X).$$
Note that both $\mu_1$ and $\mu_2$ are probability measures on $S^\ast X$, they cannot be equal. This finishes the proof of the following proposition:
\begin{prop}
If $G_\omega>0$, then quantum unique ergodicity fails for $\{\phi_j\}_{j=0}^\infty$.
\end{prop}
We remark that currently we do not have the quantum ergodicity theorem for the twisted Laplacian yet. It is unclear to us what is the correct candidate for the semiclassical measure $\mu$ of a density one subsequence of $\phi_j$ which must have $\Im r_j\to0$ and thus by \eqref{e:defect-invariance},
$$\frac{d}{dt}\varphi_t^\ast\mu=-2\omega\mu.$$

\subsection{Discussion on the pressure, the stable norm and the essential spectral gap}
\label{s:discussion}
In this subsection, we discuss the essential spectral gap $G_\omega$ and its relation to the pressure and the stable norm. 

First, if there is some $\gamma\in\mathcal{G}(X)$ such that
\begin{equation}
\label{e:essgap-sufficient}
\int_{\gamma}\omega>\frac{3}{2}\ell_\gamma
\end{equation}
then by definition \eqref{e:stablenorm}, $\|\omega\|_s>\frac{3}{2}$ and thus by \eqref{ineqpress},
$$2\|\omega\|_s-\Press(\omega)-\frac{1}{2}
>\|\omega\|_s+1-\Press(\omega)\geq0.$$
This proves Theorem \ref{t:nonque} from \eqref{e:essgap-lbweak}. 

Now we discuss the different lower bounds \eqref{e:essgap-lbweak}, \eqref{e:essgap-lb} and the arithmetic case \eqref{e:essgap-arithmetic} for $G_\omega$. In general, it is not clear which one is better. Let us consider the situation of a family of harmonic 1-forms $\{t\omega\}_{t>0}$ for some fixed $\omega\in\mathcal{H}^1(X;\mathbb{R})$.
When $t\to+\infty$, we have the following relation:
\begin{prop}
\label{p:limitpress}
For any non-zero $\omega \in \mathcal{H}^{1}(X,\mathbb{R})$,
\begin{equation}
\label{e:limitpress}
\lim_{t \to \infty}\mathrm{Pr}(t\omega)-t\|\omega\|_{s}=0.
\end{equation}
\end{prop}
\begin{proof}
We define for $\omega\in\mathcal{H}^1(X;\mathbb{R})$ and $\alpha\in[-\|\omega\|_s,\|\omega\|_s]$, 
$$H(\alpha;\omega):=\sup_{\mu \in \mathcal{M}}\left\{h_{\KS}(\mu)
\left|\int_{S^\ast X}\omega d\mu =\alpha\right.\right\}.$$
Now we fix a non-zero $\omega \in \mathcal{H}^{1}(X,\mathbb{R})$ and write $H(\alpha)=H(\alpha;\omega)$. Babillot--Ledrappier \cite{BL} showed that $H(\alpha)$ is continuous and strictly concave for $\alpha\in[-\|\omega\|_{s},\|\omega\|_{s}]$, and achieves the maximal value at $\alpha=0$. Furthermore, Anantharaman \cite{naliniearly} proved that 
$$H(\|\omega\|_{s})=H(-\|\omega\|_{s})=0.$$
We denote $\mu_t$ to be the equilibrium measure of $t\omega$, then
$$t\|\omega\|_{s}\leq \Press(t\omega)=h_{\KS}(\mu_{t})+\int_{S^\ast X}t\omega d\mu_{t}\leq t\|\omega\|_{s}+1.$$
Therefore
$$\int_{S^\ast X}\omega d\mu_t\geq\|\omega\|_s-\frac{1}{t}h_{\KS}(\mu_t)\geq\|\omega\|_s-\frac{1}{t}.$$
For $t>\|\omega\|_s^{-1}$, since $H$ is strictly decreasing on $[0,\|\omega\|_s]$, we have
$$h_{\KS}(\mu_{t})\leq H\left(\int_{S^\ast X}\omega d\mu_{t}\right)< H\left(\|\omega\|_{s}-\frac{1}{t}\right).$$
Now by the continuity of $H$, 
$$\limsup_{t\to \infty}h_{\KS}(\mu_{t})\leq H(\|\omega\|_{s})=0.$$
Thus we have 
$$\lim_{t \to \infty}\mathrm{Pr}(t\omega)-t\|\omega\|_{s}=\lim_{t \to \infty}h_{\KS}(\mu_{t})=0.$$
\end{proof}

Proposition \ref{p:limitpress} shows that
\begin{itemize}
\item As $t\to+\infty$, \eqref{e:essgap-lbweak} and \eqref{e:essgap-lb} agrees:
$$\left[2\|\omega\|_s-\Press(\omega)-\frac{1}{2}\right]
-\left[\frac{3}{2}\Press(2\omega)-2\Press(\omega)-\frac{1}{2}\right]=2\|\omega\|_s+\Press(\omega)-\frac{3}{2}\Press(2\omega),$$
and by \eqref{e:limitpress}
$$\lim_{t\to+\infty}2t\|\omega\|_s+\Press(t\omega)-\frac{3}{2}\Press(2t\omega)=0.$$
Similarly, they all agree to $\Im r_0:=\Press(t\omega)-\frac{1}{2}$ when $t\to+\infty$. However, it is not clear to us which one of \eqref{e:essgap-lbweak} and \eqref{e:essgap-lb} is better even if $t$ is large enough.
\item In the arithmetic case, when $t$ is large enough, both \eqref{e:essgap-lbweak} and \eqref{e:essgap-lb} are better than \eqref{e:essgap-arithmetic}: For example,
$$\left[2\|\omega\|_s-\Press(\omega)-\frac{1}{2}\right]-\left[\Press(\omega)-\frac{5}{4}\right]=\frac{3}{4}-2(\Press(\omega)-\|\omega\|_s).$$
By \eqref{e:limitpress},
$$\lim_{t\to+\infty}\frac{3}{4}-2(\Press(t\omega)-t\|\omega\|_s)=\frac{3}{4}.$$
\item On the other hand, to make \eqref{e:essgap-arithmetic} non-trivial, we only need that for some $\gamma\in\mathcal{P}(X)$, 
$$\int_\gamma\omega>\frac{5}{4}\ell_\gamma,$$
so that $\Press(\omega)\geq\|\omega\|_s>\frac{5}{4}$. This is better than \eqref{e:essgap-sufficient}.
\end{itemize}

We make the following conjecture motivated by the Jakobson--Naud conjecture \cite{jakobsonnaud} on resonances for convex co-compact hyperbolic surfaces:
\begin{conj}
For any $\omega\in\mathcal{H}^1(X)$ on a compact hyperbolic surfaces $X$,
\begin{equation}
\label{e:essgap-conj}
G_\omega\geq\frac{1}{2}(\Press(2\omega)-1).
\end{equation}
In particular, $G_\omega=0$ if and only if $\omega=0$. Thus for any non-zero $\omega\in\mathcal{H}^1(X;\mathbb{R})$ we have the failure of asymptotic version of Riemann hypothesis for $Z_\omega$ and the failure of quantum unique ergodicity for $\phi_j$.
\end{conj}

Finally we briefly discuss the upper bound on the essential spectral gap:
\begin{itemize}
\item A trivial upper bound is given by 
$$G_\omega\leq \Im r_0=\Press(\omega)-\frac{1}{2}.$$
It is not known how to improve this bound, one probably need Dolgopyat's method as in the work of Naud \cite{naud} and Dyatlov--Jin \cite{regfup} on resonances for convex co-compact hyperbolic surfaces.
\item From Lebeau \cite{lebeau}, we can deduce that
$$G_\omega\leq\|\omega\|_s.$$
A possible improvement of this upper bound may come from the fractal uncertainty principle of Bourgain--Dyatlov \cite{fullgap} as in Jin \cite{dwefup} for damped wave equation. We leave these questions for future papers to explore.
\end{itemize}


  
\def\arXiv#1{\href{http://arxiv.org/abs/#1}{arXiv:#1}}

\end{document}